\documentclass[12pt,reqno,a4paper]{amsart}
\usepackage{amsfonts,amsmath,times}
\usepackage{a4wide}
\usepackage{tikz}

\definecolor{com}{rgb}{0,0.3,0.75} 
\definecolor{blau}{rgb}{0,0,0.75} 
\usepackage[pdftex]{hyperref}
\hypersetup{colorlinks,linkcolor=blau,citecolor=blue,urlcolor=blau}

\allowdisplaybreaks

\newtheorem{theorem}{Theorem}
\newtheorem{lemma}{Lemma}

\newtheorem{prop}{Proposition}
\newtheorem{coroll}{Corollary}
\theoremstyle{definition}

\newtheorem{example}{Example}

\newcommand{\JAP}{\emph{Journal of Applied Probability}}

\newcommand{\PTRF}{\emph{Probability Theory and Related Fields}}

\newcommand{\mom}{\text{model}\,\ensuremath{\mathcal{M}}}
\newcommand{\mor}{\text{model}\,\ensuremath{\mathcal{R}}}

\newcommand{\momp}{\text{model}\,\ensuremath{\mathcal{M}}\,}
\newcommand{\morp}{\text{model}\,\ensuremath{\mathcal{R}}\,}

\newcommand{\fallfak}[2]{\ensuremath{#1^{\underline{#2}}}}

\newcommand{\N}{\ensuremath{\mathbb{N}}}

\newcommand{\Gro}{\ensuremath{\mathcal{O}}}
\def\P{{\mathbb {P}}}
\def\E{{\mathbb {E}}}
\def\V{{\mathbb {V}}}

\newcommand\OL{\Gro_{{\mathcal{L}}_1}}

\newcommand{\OO}{\ensuremath{\mathcal{O}}}

\newcommand{\as}{\ensuremath{\xrightarrow{(a.s.)}}}

\newcommand{\calW}{\ensuremath{\mathcal{W}}}
\newcommand\backward{\nabla}
\newcommand\field{\mathbb{F}}
\newcommand\indicator{\mathbb{I}}
\newcommand\given{\, \vert \, }
\newcommand\normal{\mathcal{N}}
\newcommand\Bigiven{\, \Big\vert \, }
\newcommand\inprob{\buildrel {\rm P} \over  \longrightarrow}
\newcommand\almostsure{\buildrel {\rm a.s} \over  \longrightarrow}
\newcommand\inL{\buildrel {{\rm L}_1} \over  \longrightarrow}
\newcommand\convD{{\buildrel  {\rm D} \over  \longrightarrow}}

\newcommand\matA{{\bf A}}
\newcommand\matM{{\bf M}}

\newcommand\Polya{P\' olya}

\newcommand{\tableT}{\rule{0pt}{2.5ex}}
\newcommand{\tableB}{\rule[-1.2ex]{0pt}{0pt}}

\author[M.~Kuba]{Markus Kuba}
\address{Markus Kuba\\
Institute of Applied Mathematics and Natural Sciences\\
University of Applied Sciences - Technikum Wien\\
H\"ochst\"adtplatz 5, 1200 Wien} 
\email{kuba@technikum-wien.at}

\author[H.~M.~Mahmoud]{Hosam M.~Mahmoud}
\address{Hosam M.~Mahmoud\\
Department of Statistics\\
The George Washington University, Washington, D.C. 20052, U.S.A.}
\email{hosam@gwu.edu}
\title[Affine urn models]{Two-colour balanced affine urn models with multiple drawings~I: central limit theorems}
\keywords{Urn model, random structure, martingale, central limit theorem}%
\subjclass[2000]{60C05, 60F05, 60G42}        
\begin{document}
\begin{abstract}
This is a research endeavor in two parts. We study a class of 
balanced urn schemes on balls of two colours (say white and black).
At each drawing, a sample of size $m\ge 1$ is drawn from the urn, and ball addition rules are applied. 
We consider these multiple drawings under sampling with or without replacement.
We further classify ball addition matrices according to the structure of the expected value 
into affine and nonaffine classes. We give a necessary and sufficient condition for a scheme to be in the affine subclass. 
For the affine subclass, we get explicit results for the expected value and second moment
of the number of white balls after $n$ steps and an asymptotic expansion of the variance. Moreover, we uncover a martingale structure, amenable to a central limit theorem formulation.
This unifies several earlier works focused on special cases of urn models with multiple drawings~\cite{ChenKu2013+,ChenWei,KuMaPan2013+,Mah2012,Moler,TsukijiMahmoud2001}. 
The class is parametrized by $\Lambda$, specified by the ratio of the two eigenvalues of a ``reduced'' ball replacement matrix and the sample size.
We categorize the class into small-index urns ($\Lambda <\frac 1 2$), critical-index urns ($\Lambda = \frac 1 2$), and
large-index urns ($\Lambda > \frac 1 2$), and triangular urns.
In the present paper (Part I), 
we obtain central limit theorems for small- and critical-index urns and prove almost-sure convergence for triangular and large-index urns. 
In a companion paper (Part II), we discuss the moment structure of large-index urns and triangular urns.
\end{abstract}

\date{\today}
\maketitle
\section{Introduction}
Urn schemes are simple, useful and versatile mathematical tools for modeling many
evolutionary processes in diverse applications such as algorithmics,  
genetics, epidimiology, physics, engineering, economics, networks (social and other types), and many more. Modeling via urns is centuries old, but perhaps the earliest  contributions in the flavor commonly called P\' olya urns (the subject of the present paper) are~\cite{Eggenberger1923,Ehrenfest}. In the first of these two classics, urns were intended to model the diffusion of gases. In the second, urns were meant
to model contagion. Many \Polya\ urn models useful for numerous applications were added later on.
In fact, they are too many (literally hundreds) to be listed individually. The sources~\cite{JohnsonKotz1977,Kotz1997} are
classic surveys listing many of these applications; 
see also~\cite{Mah2008}, where two chapters are devoted to applications in algorithmics and biosciences. 

While the term ``\Polya\ urn'' refers to a vast variety of schemes,
there is a common thread among most of them. 
Urns of the classic flavor on two colours (say white and black) evolve
in the following way. At the beginning, time zero,
the urn contains a certain number of  white and black balls. 
Thereafter, evolution of the urn occurs in discrete time steps. 
At every step, 
a ball is chosen at random from the urn. The colour of the ball is
inspected, then the ball is reinserted in the urn. According
to the colour of the sampled ball, other balls are added/removed following certain 
rules---if we have chosen a white ball, we put in
the urn $a$ white balls and $b$ black balls, but if we have chosen
a black ball, we put in the urn $c$ white balls and $d$ black
balls. The values $a, b, c, d \in \mathbb{Z}$ are fixed. 
The urn model is specified by the $2\times 2$ ball
replacement matrix $\matM = 
\begin{pmatrix} a & b \\ 
                         c &d
 \end{pmatrix}$. One is usually interested in the number 
of white balls $W_n$ after~$n$ draws, and the number of black balls~$B_n$ after~$n$ draws.
\subsection{\Polya\ urn models with multiple drawings} 
In the classic version of \Polya\ urns, one ball is sampled at each unit of (discrete) time.
The present work is devoted to the study of a generalization of the 
\Polya\ urn model, 
where \emph{multiple} balls are drawn at each discrete time step, 
their colours are inspected, then the sample is reinserted in 
the urn. Additions and deletions take place according to the drawn sample (multiset).
Such urn models recently received attention in the literature, see 
for example~\cite{ChenKu2013+,ChenWei,JohnsonKotzMahmoud2004,KuMaPan2013+,Mah2012,Moler,Renlund,TsukijiMahmoud2001}. 
The addition/removal of balls depends on the combinations of colours of the drawn balls. 
We use the notation  $\{W^kB^{m-k}\}$
to refer to a sample of size $m$ containing $k$ white balls and $m-k$ black balls.
Specifically, we draw $m\ge 1$ balls and add/remove white and black balls 
according to the multiset $\{W^kB^{m-k}\}$ of observed colours: If we draw $k$ white and $m-k$ black balls, we add $a_{m-k}$ white and $b_{m-k}$ black balls, $0\le k\le m$.
The ball replacement matrix of this urn model with multiple drawings is a rectangular $(m+1)\times 2$ matrix:
\begin{equation}
\label{MuliDrawsLinMatrix}
    \matM =
    \begin{pmatrix}
    a_0& b_0   \\
    a_1  & b_1\\
    \vdots&\vdots\\
    a_{m-1}   & b_{m-1}\\
    a_m&  b_m  \\
    \end{pmatrix}.
\end{equation}
We assume throughout that the urn model is \emph{balanced}, 
such that the overall number of added/ 
removed balls is a constant $\sigma$, independent of the composition of the sample: $a_k+b_k=\sigma\ge 1$, $0\le k\le m$. Moreover, we are only interested in so-called tenable urn models, where the process of drawing and replacing balls can be continued ad infinitum.
Several of the afore-mentioned works on urn models 
with multiple drawings were only concerned with a specific urn model. 
This includes an urn model related 
to logic circuits~\cite{Moler,TsukijiMahmoud2001}, the generalized \Polya-Eggenberger 
urn~\cite{ChenKu2013+,ChenWei}, and the generalized 
Friedman urn~\cite{KuMaPan2013+}. In this work, 
we unify and generalize these earlier works. 
We do so by discussing a more general model encompassing 
all the previously mentioned specific urns. 
\subsection{Plan of the paper and notation}
The main ingredient for our analysis is to specify all $(m+1)\times 2$ 
ball replacement matrices for which the conditional 
expectation of the number of white balls $W_n$ after $n$ draws has an affine structure of the form
$$\E\bigl[W_n \given \field_{n-1}\bigr]= \alpha_n W_{n-1} +\beta_n,\qquad n\ge 1,$$
for certain deterministic sequences $\alpha_n,\beta_n$, where $\field_n$ 
denotes the $\sigma$-algebra generated by the first~$n$ draws from the urn. 
So, we are considering a class of two-colour balanced tenable affine urns, grown under
sampling multisets. Beside such characterization, we also present a central limit theorem for $W_n$ for urns in this class with small 
and critical index, a parameter that will be defined in the sequel.
We shall return soon in a companion paper~\cite{KuMaPan201314} to deriving more families of limit laws concerning urns in the class completing the analysis of limit laws. In particular, we discuss urn models with a large index and triangular urns, using the so-called method of moments. 

We denote by 
$\fallfak{x}{k}$ the $k$th falling factorial, $x(x-1)\dots (x-k+1)$, $k\ge 0$, with $\fallfak{x}0=1$.
We shall also use $\backward$, the backward difference operator, defined by $\backward h_n = h_n -h_{n-1}$, when acting on a function $h_n$.
\section{Preliminaries}
\label{Sec:Prelim}
\subsection{Sampling schemes}
\label{Sec:PrelimSampling}
Assume that an urn contains $w$ white and $b$ black balls. We consider two different sampling schemes for drawing the $m$ balls at each step: \mom\ and \mor. In \mom\ we draw the $m$ balls without replacement. The $m$ balls are drawn at once and their colours
are examined. 
After the sample is collected, we put the entire sample back in the urn and execute the replacement rules according to the counts of colours observed. 
The tenability assumption implies that for \mom\ the coefficients $a_k$ of the 
ball replacement matrix~\eqref{MuliDrawsLinMatrix} 
satisfy the condition 
$a_k\ge -(m-k)$,\footnote{These assumptions can be relaxed a little bit,
if the initial values $W_0$ and $B_0$ are adapted to the entries in the ball replacement matrix. E.g., for $m=1$, the urn model with ball replacement matrix
$\left(\begin{smallmatrix}-3 & 8\\
6& -4
\end{smallmatrix}\right)$ is still tenable, if $W_0$ is a multiple of $3$ and $B_0$ a multiple of $4$.} for $0\le k\le m$. 

The probability $\P(W^kB^{m-k})$ of drawing $k$ white and $m-k$ black balls is given by 
$$\P(W^kB^{m-k})= \frac1{\fallfak{(b+w)}{m}}\binom{m}{k}\fallfak{w}{k}\, \fallfak{b}{m-k}=\frac{\binom{w}k\binom{b}{m-k}}{\binom{b+w}m},\qquad 0\le k\le m.$$
Thus $X$, the number of white balls in the sample,
follows a  hypergeometric distribution,
with parameters $w+b, w$, and $m$, that is, 
one that counts the number of white balls in a sample of size~$m$ balls 
taken out of an urn containing $w$ white and $b$ black balls (a total of
$\tau = w+b$ balls).
The expected value and second moment are given by
$$\E[X]=m\frac w \tau,\qquad \E[X^2] = \frac{w(w-1)m(m-1)}{\tau(\tau-1)}
         + \frac{w m}\tau.$$

In \mor, we draw the $m$ balls with replacement. The $m$ balls are drawn one at a time. After a ball is drawn,  
its colour is observed, and is reinserted in the urn, and thus 
it might reappear in the sampling of
one multiset. After $m$ balls are collected in this way (and they are all back in the urn),
we execute the replacement rules according to the counts of colours observed.
By the tenability assumption $a_k\ge -1$, for $0\le k\le m-1$ and $a_m\ge 0$, for \mor.

The probability $\P(W^kB^{m-k})$ of drawing $k$ 
white and $m-k$ black balls is given by 
$$\P(W^kB^{m-k}) = \frac1{(b+w)^m}\binom{m}{k}w^k \, b^{m-k},\qquad 0\le k\le m.$$
In other words, under \morp, the number of white balls in the multiset 
of size $m$ follows a binomial distribution with parameters $m$, and ${w}/{\tau}$, one that counts the number of successes in~$m$ 
independent identically distributed experiments, with  $w/\tau$ probability 
of success per experiment.
Let $Y$ denote such a binomially distributed 
random variable. 
The expected value and second moment are given by
$$\E[Y]=m\frac w \tau,\qquad \E[Y^2]= m\frac w \tau\Big(1-\frac w \tau\Big) + m^2\frac{w^2}{\tau^2}.$$
\subsection{Stochastic recurrence}
We start with $W_0$ white and $B_0$ black balls, $W_0,B_0\in\N_0$ assuming that $W_0+B_0\ge m$, 
to enable at least the first draw. Thereafter, tenability guarantees the perpetuation of drawing.
We are interested in the distribution of the numbers $W_n$ and $B_n$ of white and black balls after $n$ draws, respectively. 
We denote by
$$T_n=W_n+B_n,\qquad n\ge 0,$$
the total number of balls contained in the urn after $n$ draws. As 
we are considering a class of balanced urns, the total number of 
balls $T_n$ after $n$ draws is a deterministically linear:
$$T_n= \sigma n + T_0, \qquad n\ge 0.$$
We restrict ourselves to the case where the total number of balls 
increases after each draw, in other words we consider $\sigma \ge 1$.

In what follows, 
we use the notation 
$\mathbb{I}_{n}(W^kB^{m-k})$ to stand for the indicator
of the event that the multiset $\{W^kB^{m-k}\}$ is drawn in the
$n$th sampling.
Conditioning on the composition of the urn after $n-1$ draws, we obtain a stochastic recurrence for $W_n$.
The number of white balls after~$n$ draws is the number of white balls after $n-1$ draws, plus the contribution of white balls
after the $n$th sample is obtained:
\begin{equation}
\label{MuliDrawsLinDistEqn1}
W_{n} =  W_{n-1}  + \sum_{k=0}^{m} a_{m-k} \, \mathbb{I}_{n}(W^kB^{m-k}),\qquad n\ge 1.
\end{equation}
Let $\field_{n-1}$ denote the $\sigma$-field generated by the first $n-1$ draws. 
For $0\le k\le m$ the indicator variables $\mathbb{I}_{n}(W^kB^{m-k})$ satisfy 
\begin{equation}
\label{MuliDrawsLinDistEqn2}
\P\bigl(\mathbb{I}_{n}(W^kB^{m-k})=1\given \field_{n-1}\bigr)=\frac{\binom{W_{n-1}}{k}\binom{B_{n-1}}{m-k} }{\binom{T_{n-1}}{m}}=\frac{\binom{W_{n-1}}{k}\binom{T_{n-1}-W_{n-1}}{m-k} }{\binom{T_{n-1}}{m}}
\end{equation}
for \mom, and 
\begin{equation}
\label{MuliDrawsLinDistEqn3}
\P\bigl(\mathbb{I}_{n}(W^kB^{m-k})=1\given \field_{n-1}\bigr)=\binom{m}{k}\frac{W_{n-1}^k B_{n-1}^{m-k} }{T_{n-1}^m}=\binom{m}{k}\frac{W_{n-1}^k(T_{n-1}-W_{n-1})^{m-k} }{T_{n-1}^{m}}
\end{equation}
for \mor. We obtain for $W_n^{s}$, $s\ge 1$, a stochastic recurrence by taking the $s$th power 
of~\eqref{MuliDrawsLinDistEqn1}, and using the fact that the indicator variables are mutually exclusive:
\begin{equation}
\label{MuliDrawsLinDistEqn4}
W_{n}^s =  \sum_{\ell=0}^{s}\binom{s}\ell W_{n-1}^{s-\ell}\sum_{k=0}^{m} a_{m-k}^{\ell} \, \mathbb{I}_{n}(W^kB^{m-k})
,\qquad n\ge 1.
\end{equation}
\section{Affine expectation}
We classify ball replacement matrices according to the structure of the conditional expected value. 
Our motivation is that all previously treated specific urn models 
with multiple drawings~\cite{ChenKu2013+,ChenWei,KuMaPan2013+,Mah2012,Moler,TsukijiMahmoud2001}
had one feature in common, namely a simple recurrence 
relation for the conditional expectation of an affine form
$$\E\bigl[W_n \given \field_{n-1}\bigr]= \alpha_n W_{n-1} +\beta_n,\qquad n\ge 1,$$
where $\alpha_n$ and $\beta_n$ are certain deterministic sequences.  
It is desired to unify all the earlier special cases into a single simple model, and find a more general theory to work as an umbrella for these special cases and other special cases that may be equally important in application. 
In~\cite{Mah2012}, a characterization of all ball replacement matrices 
giving rise to an affine linear conditional expected value was given for the case of drawing $m=2$ balls, 
under sampling without replacement. 
We extend this analysis in the next subsection to arbitrary $m\ge 1$,
for both sampling models and characterize all ball replacement matrices
leading to such a simple relation. (Note that our results stay valid for $m=1$; here our model reduces to ordinary balanced urn models.)
Subsequently, this allows us to obtain closed formul\ae{}\, for the expected value and second moment, and to uncover an associated martingale structure. Later on, this is exploited to obtain limit theorems.
\subsection{A necessary and sufficient condition for average affinity}
We obtain, for $0\le k\le m$, a necessary and sufficient condition on the numbers $a_k,b_k$
for the conditional expectation to take an affine form, reducing the number of significant 
parameters to three: $a_{m-1}$, $a_m$ and the balance~$\sigma$.
\begin{prop}
\label{MuliDrawsLinPropLinear}
Suppose we are given the numbers $a_{m-1}$ and $a_m$, and the balance factor
$\sigma=a_k+b_k\ge 0$. For both sampling schemes, the random variable $W_n$ satisfies a linear affine relation of the form 
$$\E\bigl[W_n \given \field_{n-1}\bigr] = \alpha_n W_{n-1} +\beta_n,\qquad n\ge 1,$$
if and only if, for $0\le k\le m$, 
the numbers $a_k$ satisfy the condition
$$a_k=(m-k)a_{m-1}-(m-k-1)a_m.$$
Equivalently, the coefficients $a_k$ themselves satisfy an affinity condition: 
$$a_k= a_0 + h k,$$ with $h$ (and $h=\frac{a_m-a_0}{m}$) an integer guaranteeing tenability.
The sequences $\alpha_n$ and $\beta_n$ are given in terms of $a_{m-1},a_m$ and $T_n$ by
$$\alpha_n= \frac{T_{n-1}+m(a_{m-1}-a_m)}{T_{n-1}}, \qquad\text{and}\qquad \beta_n=a_m, \qquad n\ge 1.$$
\end{prop}
For technical reasons we assume from this point on that for balanced affine urn models the factors $\alpha_n$, as stated in Proposition~\ref{MuliDrawsLinPropLinear}, satisfy $\alpha_n>0$ for $n\ge 1$.
Equivalently, we make the assumption $T_0+m(a_{m-1}-a_m)>0$. In view of tenability and the steady increase of balls ($\sigma\ge 1$) this is a natural assumption and not really a restriction.
If for a certain model $T_0+m(a_{m-1}-a_m)\le 0$, after only a few draws (say $j_0\ge 1$), we will have $T_{j_0}+m(a_{m-1}-a_m) > 0$. 
We then restart the urn and take $j_0$ as the new beginning of time.

%

An immediate consequence of the affinity condition is the appearance of a martingale, and simple 
closed formul\ae{}\, for the expected value and the variance. 
Moreover, by appropriate choices of
the parameters $a_{m-1},a_m$ and the balance factor $\sigma$, the affinity condition covers many of the previously treated specific urn models with multiple drawings.
\begin{example}
Let $a_m=a_{m-1}=c$. We obtain $a_k=c$ for $0\le k\le m$, such that the random variable $W_n$ degenerates to a deterministic value: $W_n=W_0 + nc$.
\end{example}
\begin{example}
For $m=2$, we obtain the condition $a_0-2a_1+a_2=0$; this affinity
condition is discussed in~\cite{Mah2012}, which only considers \mom. 
\end{example}
\begin{example}
For $a_m=mc$, $a_{m-1}=(m-1)c$ and $\sigma = mc$, we obtain the generalized Friedman urn model 
with $a_k=kc$, as discussed in~\cite{KuMaPan2013+} under both sampling schemes.
\end{example}
\begin{example}
For $a_m=0$, $a_{m-1}=c$ and $\sigma=mc$, we obtain the generalized \Polya\ urn model 
with $a_k=(m-k)c$, as discussed in~\cite{ChenKu2013+,ChenWei}. 
\end{example}
\begin{example}
For $a_m=1$,  $a_{m-1}=0$ and $\sigma = 1$, we obtain $a_k=-(m-k)+1$, 
an urn model for logic circuits treated in~\cite{Moler,TsukijiMahmoud2001}.
\end{example}

In order to prove Proposition~\ref{MuliDrawsLinPropLinear}, 
we first determine the general structure of the conditional expectation.
\begin{lemma}
\label{MuliDrawsLinLemma1}
For both sampling schemes, the conditional expected value of the random 
variable~$W_n$ is a polynomial of degree $m$ (the sample size)
in~$W_{n-1}$:
$$\E\bigl[W_n \given \field_{n-1}\bigr] = \sum_{i=0}^{m}f_{n,i} W^i_{n-1},\qquad n\ge 1.$$
The values $f_{n,i}$ are model dependent. For \morp, we get 
$$f_{n,i}=\delta_{i,1}+\frac{(-1)^i}{T_{n-1}^i}\sum_{k=0}^{i}a_{m-k}\binom{m}k\binom{m-k}{m-i}(-1)^k.$$
For \momp, we get
$$f_{n,i}=\delta_{i,1}+\frac{1}{\fallfak{T_{n-1}}m}\sum_{j=0}^{m}
         \fallfak{T_{n-1}}{j}[x^i]\, p_{m,j}(x),$$
where the polynomials $p_{m,j}(x)$ are, for $0\le j \le m$, given by
$$p_{m,j}(x)= \sum_{k=0}^{m-j} a_{m-k}\binom{m}k \fallfak{x}{k}\binom{m-k}j\fallfak{(-x)}{m-k-j}.$$
\end{lemma}
\begin{proof}
Our starting point is the relation
$$\E\bigl[W_n \given \field_{n-1} \bigr] =W_{n-1} 
        +  \sum_{k=0}^{m} a_{m-k} \, \E\big[\mathbb{I}_{n}(W^kB^{m-k}) \given \field_{n-1}\big].$$
We discuss first the proof for \morp, 
which is simpler. According to~\eqref{MuliDrawsLinDistEqn3} we get
$$\E\bigl[W_n \given \field_{n-1}\bigr]=W_{n-1} +  \sum_{k=0}^{m} a_{m-k} \binom{m}{k}\frac{W_{n-1}^k(T_{n-1}-W_{n-1})^{m-k} }{T_{n-1}^{m}}.$$
Expanding $(T_{n-1}-W_{n-1})^{m-k}$ by the binomial theorem, 
and changing the order of summation yields 
$$\E\bigl[W_n \given \field_{n-1}\bigr]=W_{n-1} +  \frac{1}{T_{n-1}^m}\sum_{i=0}^m T_{n-1}^{m-i}\, W_{n-1}^i(-1)^i\sum_{k=0}^{i}a_{m-k}\binom{m}k\binom{m-k}{m-i}(-1)^k.$$
Consequently, the conditional expectation satisfies the equation
$$\E\bigl[W_n \given \field_{n-1}\bigr]=W_{n-1} +  \sum_{i=0}^m \frac{(-1)^i}{T_{n-1}^{i}}\, W_{n-1}^i\sum_{k=0}^{i}a_{m-k}\binom{m}k\binom{m-k}{m-i}(-1)^k,$$
which gives the claimed formula for $f_{n,i}$.

For \momp, from~(\ref{MuliDrawsLinDistEqn2}) we have
\begin{align*}
\E\bigl[W_n \given \field_{n-1}\bigr]&=W_{n-1} +   \frac{1}{\fallfak{T_{n-1}}m}\sum_{k=0}^{m} a_{m-k}\binom{m}k \fallfak{W_{n-1}}{k}\fallfak{(T_{n-1}-W_{n-1})}{m-k}.
\end{align*}
Next, we use the binomial theorem for the falling factorials to obtain
$$\E\bigl[W_n \given \field_{n-1}\bigr]=W_{n-1} +  \frac{1}{\fallfak{T_{n-1}}m}\sum_{k=0}^{m} a_{m-k}\binom{m}k \fallfak{W_{n-1}}{k}\sum_{j=0}^{m-k}\binom{m-k}j\fallfak{T_{n-1}}{j}\fallfak{(-W_{n-1})}{m-k-j}.$$
Changing the order of summation gives
$$\E\bigl[W_n \given \field_{n-1}\bigr]=W_{n-1} +  \frac{1}{\fallfak{T_{n-1}}m}\sum_{j=0}^{m}\fallfak{T_{n-1}}{j} \sum_{k=0}^{m-j} a_{m-k}\binom{m}k \fallfak{W_{n-1}}{k}\binom{m-k}j\fallfak{(-W_{n-1})}{m-k-j}.$$
The inner sum on the right-hand side is exactly the polynomial $p_{m,j}(W_{n-1})$. The polynomials can be expanded into 
powers of $W_{n-1}$, leading to the stated result.
\end{proof}
\begin{proof}[Proof of Proposition~\ref{MuliDrawsLinPropLinear}]
Given the numbers $a_{m-1}$ and $a_m$, 
we need to ensure that the conditional expected value of $W_n$ only involves $W_{n-1}$ and constants, but no higher powers of $W_{n-1}$.
By Lemma~\ref{MuliDrawsLinLemma1}, this is equivalent to the condition $f_{n,i}=0$, $2\le i\le m$.
It remains to show that this condition is fulfilled, if and only if the 
coefficients of a ball replacement matrix satisfy the stated condition $a_k=(m-k)a_{m-1}-(m-k-1)a_m$. 
Note that by collecting the coefficient~$k$ and expressing $a_{m-1}$ in terms of $a_0=m(a_{m-1}-a_m)+a_m$, we have the equivalent condition
$a_k= h k+ a_0,$ with arbitrary $a_0$ and $h$ satisfying tenability.

We start with \mor. By Lemma~\ref{MuliDrawsLinLemma1} the condition $f_{n,i}=0$, $2\le i\le m$, implies 
the following linear equations for the numbers $a_k$, $0\le k\le m-2$, independent of $T_{n-1}$ and thus independent of $n$, too:
$$\sum_{k=2}^{i}a_{m-k}\binom{m}k\binom{m-k}{m-i}(-1)^k=m\binom{m-1}{m-i}a_{m-1}-\binom{m}{m-i}a_m, \qquad 2\le i\le m.$$
This system of linear equations is upper triangular and has a unique solution. The solution can be obtained by Cramer's rule.
However, in order to avoid more involved calculations, 
we can check that the stated solution $a_k=(m-k)a_{m-1}-(m-k-1)a_m$ satisfies the equations by simple algebraic manipulations, which are omitted here. 
For \mom, by contrast to the previous case,   
the $m-1$ equations $f_{n,i}=0$, $2\le i\le m$,
are not independent of $n$, since 
they involve $T_{n-1}$: 
$$f_{n,i}=\frac{1}{\fallfak{T_{n-1}}m}\sum_{j=0}^{m}
         \fallfak{T_{n-1}}{j}[x^i]\, p_{m,j}(x),$$
$2\le i\le m$. In order to ensure that $f_{n,i}=0$ for all $n$, with $2\le i\le m$, the coefficient $[x^i]\, p_{m,j}(x)$ of the falling factorials $\fallfak{T_{n-1}}{j}$ have to vanish for all $n$. Assume conversely that there exists a largest $j=j_0$, $1\le j\le m$, such that $[x^i]\, p_{m,j}(x)\neq 0$. Then, for large $n$, we have 
$$f_{n,i}=\frac{1}{\fallfak{T_{n-1}}m}\sum_{j=0}^{j_0}
         \fallfak{T_{n-1}}{j}[x^i]\, p_{m,j}(x) 
         =\frac{1}{\fallfak{T_{n-1}}m} (\fallfak{T_{n-1}}{j_0}[x^i]\, p_{m,j_0}(x) + \mathcal{O}(T_{n-1}^{j_0-1}),$$
such that $f_{n,i} \sim \frac{T_{n-1}^{j_0}}{T_{n-1}^m} [x^i]\, p_{m,j_0}(x) \neq 0$. 
Thus, we obtain the system of equations 
$$[x^i]\, p_{m,j}(x)=[x^i] \sum_{k=0}^{m-j} a_{m-k}\binom{m}k 
           \fallfak{x}{k}\binom{m-k}j\fallfak{(-x)}{m-k-j}=0,$$
for $2\le i \le m$ and $0\le j \le m$. This leads to an overdetermined 
system of linear equations for the coefficients $a_k$. 
Instead of writing the whole system, it is sufficient to derive an 
exactly solvable subsystem of equations involving all the coefficients $a_k$, $0\le k\le m$.
In order to do so, we concentrate on the equations arising from the coefficient of $x^{m-j}$. 
This is the highest power of~$x$ in the polynomials $p_{m,j}(x)$. 
We get
\begin{align*}
[x^{m-j}]\, p_{m,j}(x) &=[x^{m-j}]\sum_{k=0}^{m-j} a_{m-k}\binom{m}k \fallfak{x}{k}\binom{m-k}j\fallfak{(-x)}{m-k-j}\\
   &=\sum_{k=0}^{m-j} a_{m-k}\binom{m}k \binom{m-k}j(-1)^{m-k-j},
\end{align*}
$0\le j\le m$. We allow $a_{m-1}$, and $a_m$ to be freely chosen. Setting $j=m-i$ leads 
to a the system of equations
for the numbers $a_k$ with $0\le k\le m-2$:
$$\sum_{k=0}^{i}a_{m-k}\binom{m}k\binom{m-k}{m-i}(-1)^{m-k}=0, \qquad 2\le i\le m.$$
This system coincides with the system of equations previously derived 
for sampling with replacement. 
It has the stated unique solution. Hence, the overdetermined system of equations 
$$[x^i]\, p_{m,j}(x)=0,\qquad 2\le i \le m, \qquad 0\le j\le m,$$
has either exactly one solution or no solution at all. 
It remains to show that coefficients satisfying the affinity condition $a_k=(m-k)a_{m-1}-(m-k-1)a_m$
lead to a solution. Starting from~\eqref{MuliDrawsLinDistEqn2} we get
$$\E\bigl[W_n \given \field_{n-1}\bigr]=W_{n-1} +  \sum_{k=0}^{m} \big(k(a_{m-1}-a_m)+a_m\big)\frac{\binom{W_{n-1}}{k}\binom{T_{n-1}-W_{n-1}}{m-k} }{\binom{T_{n-1}}{m}}.$$
Next, we use Vandermonde's convolution formula
$$\sum_{k=0}^{m}\binom{r}{k}\binom{s}{m-k}=\binom{r+s}{m},$$
and obtain 
$$\E\bigl[W_n \given \field_{n-1}\bigr] =W_{n-1} +  W_{n-1}(a_{m-1}-a_m)\frac{\binom{T_{n-1}-1}{m-1}}{\binom{T_{n-1}}{m}}+a_m.$$
\end{proof}

\subsection{Expected value and second moment}
Next, we generalize the result of Bagchi and Pal~\cite{Bagchi1985} for the expected value and the second moment, when drawing a single ball (the case $m=1$) to balanced affine urn models with multiple drawings. In order to state our result we introduce the quantity $g_n$ given by
\begin{align}
\label{MuliDrawsLinValueGn}
g_n &=\prod_{j=0}^{n-1}\frac{T_j}{T_j+m(a_{m-1}-a_m)}
= \frac{\binom{n-1+\frac{T_0}\sigma}{n}}{\binom{n-1+\frac{T_0+m(a_{m-1}-a_m)}\sigma}{n}} =\frac{\Gamma(n+\frac{T_0}\sigma)\, \Gamma(\frac{T_0+m(a_{m-1}-a_m)}\sigma)}{\Gamma(\frac{T_0}\sigma)\, \Gamma(n+\frac{T_0+m(a_{m-1}-a_m)}\sigma )}.
\end{align}\begin{prop}
\label{MuliDrawsLinProp1}
The expected value of the random variable $W_{n}$, counting the number 
of white balls in a two-color balanced affine urn model with multiple drawings, is for both sampling models 
$\mathcal{M}$ and $\mathcal{R}$ given by $\E[W_{n}]=\frac{a_m}{g_n}\sum_{j=1}^{n}g_j +W_0\frac{1}{g_n}$, with $g_n$ as given stated above in~\eqref{MuliDrawsLinValueGn}.
For $\frac{m(a_{m-1}-a_m)}\sigma<1$, we have the closed form expression
$$\E[W_{n}]= \frac{a_m(n+\frac{T_0}\sigma)}{1-\frac{m(a_{m-1}-a_m)}\sigma} + \Big(W_0-\frac{\frac{a_mT_0}\sigma}{1-\frac{m(a_{m-1}-a_m)}\sigma}\Big)\frac{\binom{n-1+\frac{T_0+m(a_{m-1}-a_m)}\sigma}{n}}{\binom{n-1
+\frac{T_0}\sigma}{n}},$$
as well as the asymptotic expansion 
\begin{align*}
\E[W_{n}] &=
 \frac{a_m \sigma }{\sigma-m(a_{m-1}-a_m)}\, n \\
 & \qquad {} + \Big( W_0-\frac{\frac{a_mT_0}\sigma}{1-\frac{m(a_{m-1}-a_m)}\sigma}\Big) \frac{\Gamma(\frac{T_0}\sigma)}{\Gamma(\frac{T_0+m(a_{m-1}-a_m)}\sigma)} \, n^{\frac{m(a_{m-1}-a_m)}\sigma} 
+ \Gro(1),
\end{align*}
Moreover, for $\frac{m(a_{m-1}-a_m)}\sigma=1$ we obtain $\E[W_{n}]=W_0\frac{n\sigma +T_0}{T_0}$.
\end{prop}
%

\begin{proof}[Proof of Proposition~\ref{MuliDrawsLinProp1}]
From Proposition~\ref{MuliDrawsLinPropLinear} we get
\begin{equation}
\label{MuliDrawsLinExpected1}
\E[W_n]=\bigg(\frac{T_{n-1}
           +m(a_{m-1}-a_m)}{T_{n-1}}\bigg) \, \E[W_{n-1}]+a_m,\qquad n\ge 1. 
\end{equation}
Multiplication with $g_n$ as defined in~\eqref{MuliDrawsLinValueGn} gives the recurrence relation
$$g_n\, \E[W_n]=g_{n-1}\, \E[W_{n-1}] + g_n a_m, $$
such that
\begin{equation}
\label{MuliDrawsLinExpectedCorrect}
\E[W_n]=\frac{a_m}{g_n}\sum_{j=1}^{n}g_j + W_0\frac{g_0}{g_n}=\frac{a_m}{g_n}\sum_{j=1}^{n}g_j +W_0\frac{1}{g_n}.
\end{equation}
Applying the summation formula 
$$\sum_{k=1}^{s}\frac{\binom{k+x}{k}}{\binom{k+y}k}
           =\frac{(s+1+y)\binom{s+1+x}{s+1}}{(x+1-y)\binom{s+1+y}s+1}
           -\frac{x+1}{x+1-y},$$
to the sum involving $g_n$, which has the form
$\binom{n+x}{n}/\binom{n+y}{n}$, with $x=\frac{T_0}\sigma-1$ and $y= \frac{T_0}\sigma-1+\frac{m(a_{m-1}-a_m)}\sigma$, 
we obtain the result, valid for $\frac{m(a_{m-1}-a_m)}\sigma<1$.
For $\frac{m(a_{m-1}-a_m)}\sigma=1$, we observe that by the tenability assumption on the urn, we obtain 
for both sampling models the conditions $a_m\ge 0$, and also $b_0\ge 0$. Thus, we get from Proposition~\ref{MuliDrawsLinPropLinear}
$$\sigma=a_0+b_0= m(a_{m-1}-a_m)+a_m+b_0\ge m(a_{m-1}-a_m),$$ 
such that $a_m=b_0=0$, and the result follows directly from~\eqref{MuliDrawsLinExpectedCorrect}. 

In order to obtain asymptotic expansions, 
we only need Stirling's formula for the Gamma function:
$$\Gamma(z) = \Bigl(\frac{z}{e}\Bigr)^{z}\frac{\sqrt{2\pi }}{\sqrt{z}}%
    \Bigl(1+\frac{1}{12z}+\frac{1}{288z^{2}}+\Gro\Bigl(\frac{1}{z^{3}}\Bigr)\Bigr), \qquad |z| \to \infty.$$
Hence, we obtain 
$$\frac1{g_n}= \frac{\Gamma(\frac{T_0}\sigma)\, \Gamma(n+\frac{T_0+m(a_{m-1}-a_m)}\sigma)}{\Gamma(n+\frac{T_0}\sigma)\, \Gamma(\frac{T_0+m(a_{m-1}-a_m)}\sigma)} 
= \frac{\Gamma(\frac{T_0}\sigma)}{\Gamma(\frac{T_0+m(a_{m-1}-a_m)}\sigma)}n^{\frac{m(a_{m-1}-a_m)}\sigma}\Big(1+\Gro\Big(\frac 1 n\Big)\Big),$$
yielding the stated result.
\end{proof}
\begin{prop}
\label{MuliDrawsLinProp2}
For balanced affine urn schemes, the second moment of $W_n$ is
$$\E[W_n^2] =
\frac{\binom{n-1+\lambda_1}{n}\binom{n-1+\lambda_2}{n}}{\binom{n-1+\frac{T_0}\sigma}{n}\binom{n-1+\frac{T_0-1}\sigma}{n}}
 \bigg(W_0^2+\sum_{j=1}^{n}\frac{(\beta_j\E[W_{j-1}]+a_m^2)\binom{j-1+\frac{T_0}\sigma}{j}\binom{j-1+\frac{T_0-1}\sigma}{j}}{\binom{j-1+\lambda_1}{j}\binom{j-1
 +\lambda_2}{j}}\bigg)$$
for \mom, with $\lambda_{1,2}=\frac{m(a_{m-1}-a_m)+T_0-\frac12 \pm \frac12 \sqrt{1+4m(a_{m-1}-a_m)(a_{m-1}-a_m+1)}}{\sigma}$,
and 
\begin{equation}
\label{MuliDrawsVarBetaModelM}
\beta_n=(a_{m-1}-a_m)^2\Big(\frac{m}{T_{n-1}}-\frac{\fallfak{m}2}{\fallfak{T_{n-1}}2}\Big) + \frac{2m a_m(a_{m-1}-a_m)}{T_{n-1}}+2a_m.
\end{equation}
For \morp, the second moment of $W_n$ is
$$\E[W_n^2]=\frac{\binom{n-1+\mu_1}{n-1}\binom{n-1+\mu_2}{n}}{\binom{n-1+\frac{T_0}\sigma}{n}^2}
\bigg(W_0^2+\sum_{j=1}^{n}\frac{(\beta_j\E[W_{j-1})+a_m^2]\binom{j-1+\frac{T_0}\sigma}{j}^2}{\binom{j-1+\mu_1}{j}\binom{j-1+\mu_2}{j}}\bigg)
\bigg),$$
with $\mu_{1,2}=\frac{m(a_{m-1}-a_m)+T_0 \pm  (a_{m-1}-a_m)\sqrt{m}}{\sigma}$, and 
\begin{equation}
\label{MuliDrawsVarBetaModelR}
\beta_n=\frac{(a_{m-1}-a_m)^2m}{T_{n-1}} + \frac{2m a_m(a_{m-1}-a_m)}{T_{n-1}}+2a_m.
\end{equation}
\end{prop}
\begin{proof}
We use the stochastic recurrence~\eqref{MuliDrawsLinDistEqn4}, with $s=2$, 
and obtain for the conditional expectation the equation
$$\E\bigl[W_n^2\given \field_{n-1}\bigr]=W_{n-1}^2 + \sum_{k=0}^{m}(2W_{n-1}a_{m-k}+a_{m-k}^2)
              \, \E\bigl[\mathbb{I}_{n}(W^kB^{m-k}) \given \field_{n-1}\big].$$
By Proposition~\ref{MuliDrawsLinPropLinear} and the affinity condition, we further get 
\begin{equation}
\begin{split}
\label{MuliDrawOmegaJZwei1}
\E\bigl[W_n^2\given \field_{n-1}\bigr]&=W_{n-1}^2 + a_m(2W_{n-1}+a_m)\\
&\qquad {}+\sum_{k=0}^{m}\big(k^2(a_{m-1}-a_m)^2+2k(a_{m-1}-a_m)(a_m+W_{n-1})\big)\\
&\qquad \qquad\qquad {} \times\E\bigl[\mathbb{I}_{n}(W^kB^{m-k}) \given \field_{n-1}\big].
\end{split}
\end{equation}
The sums depend on the particular sampling model. 
According to~\eqref{MuliDrawsLinDistEqn2}, for \momp, the number of drawn white balls 
in the sample of size $m$ is given by a hypergeometric distribution 
with parameters $T_{n-1}$, $W_{n-1}$ and $m$. Alternatively, for 
 \morp, the number of drawn white balls 
in the sample of size $m$ is given by a binomial distribution with parameters 
$m$ and $W_{n-1}/T_{n-1}$.

We take expectations and use the results of Section~\ref{Sec:Prelim} to simplify the sums. 
Consequently, we obtain for both models a linear recurrence relation of the form
$$\E[W_{n}^2] = \alpha_n\, \E[W_{n-1}^2] + \beta_n\,  \E[W_{n-1}] 
       + \gamma_n,\qquad n\ge 1,$$
with $\E[W_{n}]$ as given in Proposition~\ref{MuliDrawsLinProp1}.
For \momp, the sequences $\alpha_n$ and $\gamma_n$ are given by
$$\alpha_n=1+\frac{(a_{m-1}-a_m)^2\fallfak{m}2}{\fallfak{T_{n-1}}2}+\frac{2(a_{m-1}-a_m)m}{T_{n-1}},\qquad\gamma_n=a_m^2,$$
and $\beta_n$ as stated in~\eqref{MuliDrawsVarBetaModelM}.
For \mor, we have
$$\alpha_n=1+\frac{(a_{m-1}-a_m)^2\fallfak{m}2}{T_{n-1}^2}+\frac{2(a_{m-1}-a_m)m}{T_{n-1}},\qquad \gamma_n=a_m^2,$$
and $\beta_n$ as stated in~\eqref{MuliDrawsVarBetaModelR}. 
The recurrence relation for $\E[W_n^2]$ is readily solved in a manner similar to that we used to solve
~\eqref{MuliDrawsLinExpected1}, and we obtain 
$$\E[W_n^2] = \bigg(\prod_{\ell=1}^{n}\alpha_\ell\bigg)\bigg(W_0^2+\sum_{j=1}^{n}\frac{\beta_j\E[W_{j-1}] +\gamma_j}{\prod_{\ell=1}^{j}\alpha_\ell}\bigg),$$
with $\E[W_n]$ given by Proposition~\ref{MuliDrawsLinProp1}. Finally, we simplify the products $\prod_{\ell=1}^{n}\alpha_\ell$ by viewing $\alpha_n$ as a rational function
in the variable $n$, and factorizing it into linear terms of the forms 
$$\alpha_n=\frac{(n-1+\lambda_1)(n-1+\lambda_2)}{(n-1+\frac{T_0}\sigma)(n-1+\frac{T_0-1}\sigma)},\qquad
\text{and}\qquad
\alpha_n=\frac{(n-1+\mu_1)(n-1+\mu_2)}{(n-1+\frac{T_0}\sigma)^2},$$
for models $\mathcal{M}$ and $\mathcal{R}$, respectively.
\end{proof}

\subsection{Martingale structure}
Next, we deduce from the linear affine structure of the conditional expected value of $W_n$ and the previous result for the expected value the following result. 

\begin{prop}
\label{Prop:caligraphicW}
For balanced affine urn schemes with $a_m\neq 0$, the centered random variable 
$$\mathcal{W}_n=g_n(W_n-\E[W_n])=g_nW_n- a_m \sum_{j=1}^{n}g_j - W_0,$$ with $g_n$ as defined in~\eqref{MuliDrawsLinValueGn}, 
is a martingale with respect to the natural filtration: $\E[\mathcal{W}_n\given \field_{n-1}]=\mathcal{W}_{n-1}$, $n\ge 1$, and $\mathcal{W}_0=0$.

For balanced affine urn schemes with $a_m=0$, the random variable $\mathfrak{W}_n=g_nW_n$ is a non-negative martingale and converges almost surely to a limit $\mathfrak{W}_\infty$. 
\end{prop}
\begin{proof}[Proof of Proposition~\ref{Prop:caligraphicW}] 
By Proposition~\ref{MuliDrawsLinPropLinear} the conditional expectation is given by
\begin{equation}
\label{MuliDrawsOmegaJEins}
\E\bigl[W_n \given \field_{n-1}\bigr]=W_{n-1}\bigg(\frac{T_{n-1}+m(a_{m-1}-a_m)}{T_{n-1}}\bigg)+a_m,
\end{equation}
for $n\ge 1$. As in the proof of Proposition~\ref{MuliDrawsLinProp1},
we obtain
$$\E\bigl[g_nW_n\given \field_{n-1}\bigr]=g_{n-1}W_{n-1}+g_n a_m, \qquad n\ge 1,$$
By definition
$$\mathcal{W}_n=g_n  W_n - a_m \sum_{j=1}^{n}g_j - W_0,$$
and we get the representation
$$g_n\, \E\bigl[W_n \given \field_{n-1}\bigr]- a_m \sum_{j=1}^{n}g_j-W_0
          = g_{n-1}W_{n-1}- a_m\sum_{j=1}^{n-1}g_j-W_0,$$
such that 
$$\E\bigl[\mathcal{W}_n\given \field_{n-1}\bigr]=\mathcal{W}_{n-1}.$$
We also note that $\mathcal{W}_0=g_0(W_0-\E[W_0])=0$, and so $\E[\mathcal{W}_n]=0$. Moreover,
$$\E\bigl[|\mathcal{W}_n|\bigl]=g_n\, \E\bigl[|W_n-\E[W_n]|\bigr].$$

For $a_m=0$, we note that $W_n\ge 0$ and also $\mathfrak{W}_n=g_nW_n\ge 0$. By martingale theory, $\mathfrak{W}_n$  converges almost surely to a limit: 
$\mathfrak{W}_n\almostsure \mathfrak{W}_\infty$.
\end{proof}

\section{Limit theorems}
In this section, we discuss limit theorems for the number of white balls. 
Our limit theorems are valid for arbitrary $m\ge 1$, unifying the earlier observed phenomena 
for the case $m=1$, and covering new such cases, as well as extending the result to
larger sample size. 
For balanced urn models and a single ball in the sample,
one considers the ball replacement matrix $\matM=\begin{pmatrix}a&b\\c&d\end{pmatrix}$, 
with balance factor $\sigma$, with eigenvalues $\Lambda_1=\sigma$ and $\Lambda_2=a-c$. For this classic case,
there is a known trichotomy~\cite{Chauvin1,Chauvin2,Jan2004,Jan2006,NeiningerKnape}: 
(1) triangular urn models with a nongaussian limit for $c=0$ (or $b=0$), 
(2) the so-called \emph{small urns} with a Gaussian limit for $c> 0$ and $\Lambda_2/\Lambda_1\le \frac12$, and 
(3) the so-called \emph{large urns} with a nongaussian limit for $c> 0$ and $\Lambda_2/\Lambda_1>\frac12$.
Note that owing to the balance, the urn actually has only three parameters $a,c$ and $\sigma$.
The terms ``small urns" and ``large urn" were used by other researchers~\cite{Chauvin1}.
We prefer to think of the ratio of eigenvalues as an index and refer to urns
with small versus large index. It is the index that can be large or small,
not the physical container (urn, box, etc.).

For urn models with multiple drawings and affine expectation,
we obtain a similar characterization. By Proposition~\ref{MuliDrawsLinPropLinear}, our class of 
urns is determined by $a_{m-1},a_m$ and the balance factor~$\sigma$, satisfying
the affinity condition $a_k=(m-k)(a_{m-1}-a_m)+a_m$, $0\le k\le m$. 
We call $\matA=
\begin{pmatrix}
a_{m-1} & b_{m-1}\\
a_m&b_m\\
\end{pmatrix}$ the reduced ball replacement matrix. For the affine subclass of balanced urn models,
the eigenvalues of $\matA$ are $\Lambda_1=\sigma$ and $\Lambda_2=a_{m-1}-a_m$. 
It turns out that the behaviour of the urn critically depends on the \emph{urn index} $\Lambda$ given by 
the ratio $\Lambda_2/\Lambda_1$ of the two eigenvalues of $\matA$ times the sample size $m$:
$$\Lambda = \Lambda(m, \sigma) := \frac m \sigma (a_{m-1}-a_m).$$ 
This parameter governs the growth of the second largest term in the asymptotic expansion of the expected value.  
For instance, in terms of this index, the expectation
in Proposition~\ref{MuliDrawsLinProp1} is
$$\E[W_n] = \frac {a_m} {1-\Lambda}\, n + \Gro(n^\Lambda) + \Gro(1).$$
In the following we obtain a central limit theorem for urn models with ``small index" 
$\Lambda< \frac 1 2$ and ``critical index'' $\Lambda=\frac12$. Note that the case $\Lambda=0$ is excluded from our considerations because it leads to $a_m=a_{m-1}$ and by the affinity condition to $a_k=a_m$, $0\le k\le m$; thus we have deterministic development: $W_n=W_0 + a_m n$. We also obtain almost sure and $L_2$-convergence for ``large index'' urns $\Lambda>\frac 1 2$.
We call an urn model triangular if $a_m=0$ or $b_0=0$ (or both). We already obtained for $a_m=0$ almost sure convergence in Proposition~\ref{Prop:caligraphicW}. 
Since $B_n=T_n-W_n$, we can reduce the case $b_0=0$ and $a_m\ge 0$ by reversing the colors to $a_m=0$ and $b_0\ge 0$.
A detailed study of the moment structure of large index urns and the triangular urns with $a_m=0$ (and $b_0\ge 0$) will appear in a companion work.

\subsection{Asymptotic expansion of the variance}
An asymptotic expansion of the variance of $W_n$ can be obtained from the explicit expressions for the expected value and the second moment. It is required 
to prove later on a central limit theorem for $\Lambda\le\frac12$ and almost sure convergence of large-index urns. 
\begin{theorem}
\label{TheVar}
For balanced affine urn schemes,
the variance satisfies the following expansions:
\item Small-index urns, the case $\Lambda<\frac12$: 
  \[
	\V[W_n]=
	\frac{a_mb_0\Lambda^2}{m(1-2\Lambda)(1-\Lambda)^2}\,  n +o(n).
	\]
	\item Critical-index urns, the case $\Lambda=\frac12$: 
	\[
	\V[W_n]=\frac{a_mb_0}{m} n\log n + \OO(n),
	\]
	\item Large-index urns, the case $\Lambda>\frac12$: 
	\[
	\V[W_n]=C n^{2\Lambda} + \OO(n),
	\]
	with the constant $C$ being model-dependent given by an infinite sum:
	\begin{equation*}
	\begin{split}
C&=\frac{W_0^2}{\psi_0} +\sum_{j=1}^{\infty}\frac{\beta_j\E[W_{j-1}]+a_m^2-\psi_j2a_m j^{-\Lambda}\Big(\frac{a_m}{1-\Lambda}j^{1-\Lambda}+( W_0-\frac{\frac{a_mT_0}\sigma}{1-\Lambda}) \frac{\Gamma(\frac{T_0}\sigma)}{\Gamma(\frac{T_0}\sigma+\Lambda)}\Big)}{\psi_j}\\
&\quad+\frac{2a_m^2}{1-\Lambda}\zeta(2\Lambda-1)+2a_m\Big( W_0-\frac{\frac{a_mT_0}\sigma}{1-\Lambda}\Big) \frac{\Gamma(\frac{T_0}\sigma)}{\Gamma(\frac{T_0}\sigma+\Lambda)} \zeta(\Lambda)\\
&\quad -\Big( W_0-\frac{\frac{a_mT_0}\sigma}{1-\Lambda}\Big)^2 \frac{\Gamma^2(\frac{T_0}\sigma)}{\Gamma^2(\frac{T_0}\sigma+\Lambda)},
	\end{split}
	\end{equation*}
	with $\zeta(z)$ denoting the Riemann zeta function and $\beta_j$, $\psi_j$, $\E[W_{j-1}]$ as given in~\eqref{MuliDrawsVarBetaModelM}, \eqref{MuliDrawsVarBetaModelR},~\eqref{VarSimple2Eqn},~and Proposition~\ref{Prop:caligraphicW}.
\end{theorem}
\begin{proof}
Our starting point is the expression for $\E[W_n^2]$ in Proposition~\ref{MuliDrawsLinProp2}. In order to perform a unified analysis for the two models, we use the representation
\begin{equation}
\label{VarSimple1Eqn}
\E[W_n^2]=\frac{\psi_n}{\psi_0}  W_0^2 + \psi_n\sum_{j=1}^{n}\frac{\beta_j\E[W_{j-1}]+a_m^2}{\psi_j},
\end{equation}
with
\begin{equation}
\label{VarSimple2Eqn}
\psi_n=
\begin{cases}
\displaystyle{\frac{\Gamma(n+\lambda_1)\Gamma(n+\lambda_2)}{\Gamma(n+\frac{T_0}{\sigma})\, \Gamma(n+\frac{T_0-1}{\sigma})}}, & \mom;\\[0.35cm]
\displaystyle{\frac{\Gamma(n+\mu_1)\Gamma(n+\mu_2)}{\Gamma(n+\frac{T_0}{\sigma})^2}}, & \mor.
\end{cases}
\end{equation}
We refine our previous result and observe that the expected value $\E[W_n]$ satisfies the asymptotic expansion
\begin{equation}
\begin{split}
\label{ExpansionEn}
\E[W_{n}] &= \frac{a_m }{1-\Lambda}\, n  + \Big( W_0-\frac{\frac{a_mT_0}\sigma}{1-\Lambda}\Big) \frac{\Gamma(\frac{T_0}\sigma)}{\Gamma(\frac{T_0}\sigma+\Lambda)} \, n^{\Lambda} + \frac{T_0 a_m }{\sigma(1-\Lambda)} + O(n^{-1+\Lambda}).
\end{split}
\end{equation}
Moreover, $g_n$ satisfies the asymptotic expansion
\begin{equation}
\label{ExpansionGn}
g_n=\frac{\Gamma(\frac{T_0}\sigma+\Lambda)}{\Gamma(\frac{T_0}\sigma)}n^{-\Lambda}\Big(1+\frac1{2n}\Lambda\Big(1-\frac{2T_0}\sigma-\Lambda\Bigr)+O\Bigl(\frac1{n^2}\Bigr)\Bigr).
\end{equation}
Furthermore, $\beta_n$ satisfies for both urn models the asymptotic expansion
\begin{equation}
\label{ExpansionBn}
\beta_n=2a_m+\frac{\Lambda(a_{m-1}+a_m)}{n}  +O\Bigl(\frac1{n^2}\Bigr).
\end{equation}
We need the expansion
\begin{equation*}
\psi_n=n^{2\Lambda}\Big(1+\frac{M}{n}+O\Bigl(\frac1{n^2}\Bigr)\Big),
\end{equation*}
with the constant $M$ given by
\begin{equation*}
M=
\begin{cases}
\frac12\big(\lambda_1^2-\lambda_1+\lambda_2^2-\lambda_2
-\frac{T_0^2}{\sigma^2}+\frac{T_0}{\sigma}-\frac{(T_0-1)^2}{\sigma^2}+\frac{T_0-1}{\sigma}\big),& \mom;\\
\frac12\big(\mu_1^2-\mu_1+\mu_2^2-\mu_2
-2\frac{T_0^2}{\sigma^2}+2\frac{T_0}{\sigma}\big),& \mor,
\end{cases}
\end{equation*}
with $\lambda_i$, $\mu_i$ as given in Proposition~\ref{MuliDrawsLinProp2}.
After simplifications it turns out that for both models the constant $M$ 
is given by 
\[
M=\Lambda^2-\Lambda+\frac{\Lambda^2}{m}+\frac{2\Lambda T_0}\sigma.
\]
In order to keep track of the different expansions in a readable transparent way, we introduce a shorthand notation:
\begin{equation}
\begin{split}
\label{ExpansionSummary}
\E[W_{n}] &= E_1 n + E_2 n^{\Lambda}+E_3+O(n^{-1+\Lambda}),\\
\beta_n&=B_1+B_2n^{-1} + O(n^{-2}),\\
\end{split}
\end{equation}
with constants $E_i$, $B_i$ as given in~\eqref{ExpansionEn} and~\eqref{ExpansionBn}.
We note that
\begin{equation*}
\bigl(\E[W_n]\bigr)^2=E_1^2 n^2 + 2E_1E_2 n^{1+\Lambda}+2E_1E_3n+E_2^2n^{2\Lambda}+o(n).
\end{equation*}
We shall prove that
\[
\E[W_n^2]=E_1^2n^2 + 2E_1E_2 n^{1+\Lambda}+\varphi_n,
\]
with 
$$\varphi_n=
\begin{cases}
\Bigl(  \frac{a_m(\sigma(1-\Lambda)-a_m)\Lambda^2}{m(1-2\Lambda)(1-\Lambda)^2} + 2E_1E_3\Bigr) n +o(n), &		\text{for }\Lambda<\frac12;\\ 
\frac{a_m(\frac{\sigma}{2}-a_m)}{m}n\log n +\OO(n),&\text{for }\Lambda=\frac12;\\ 
(C+E_2^2) n^{2\Lambda}+ \OO(n),&\text{for }\Lambda>\frac12.
\end{cases}
$$
We start with the small-index urns satisfying $\Lambda<\frac12$. Assume first that $0<\Lambda<\frac12$. We postpone the remaining case $\Lambda<0$ 
to the end (note that $\Lambda=0$ leads to a degenerate urn model). The expansion of $\E[W_n^2]$ is obtained as follows. First, let $j=j(n)$, with $j\to \infty$, and write
\begin{align}
\label{ExpansionJ}
\frac{\beta_j\E[W_{j-1}]+a_m^2}{\psi_j} 
   &= B_1E_1j^{1-2\Lambda} + B_1E_2j^{-\Lambda} \nonumber\\
   & \qquad {} +(a_m^2+B_1E_3+E_1B_2-B_1E_1-B_1E_1M)j^{-2\Lambda}+O(j^{-1-\Lambda}).
\end{align}
Replacing the summands by their asymptotic expansion leads to an error of magnitude $O(1)$.
This is fully sufficient for our purpose. Consequently, we use the following identity, which can be obtained using the Euler-MacLaurin summation formula (see ~\cite{GraKnuPa}; 
Pages 595--596):  
\begin{equation}
\sum_{j=1}^{n}j^{\alpha}=\frac{n^{\alpha+1}}{\alpha+1}+\frac{n^{\alpha}}2 + \zeta(-\alpha)+ O(n^{\alpha-1}),
\label{KNUTH1}
\end{equation}
for $\alpha\neq -1$, where $\zeta(z)$ denotes the Riemann zeta function. 
We obtain the expansion
\begin{equation}
\begin{split}
\label{ExpansionSmall1}
\sum_{j=1}^{n}\frac{\beta_j\E[W_{j-1}]+a_m^2}{\psi_j}&=
B_1E_1\Big(\frac{n^{2-2\Lambda}}{2-2\Lambda}+\frac{n^{1-2\Lambda}}{2}\Big)
+B_1E_2\frac{n^{1-\Lambda}}{1-\Lambda}\\
&\quad+\big(a_m^2+B_1E_3+E_1B_2-B_1E_1-B_1E_1M\big)\frac{n^{1-2\Lambda}}{1-2\Lambda}+ O(1).
\end{split}
\end{equation}
Since $\frac{B_1E_1}{2(1-\Lambda)}=E_1^2$ and $\frac{B_1E_2}{1-\Lambda}=2E_1E_2$, 
we obtain---taking into account the expansion of $\psi_n$---the following:
\begin{equation}
\begin{split}
\label{ExpansionSmall2}
\psi_n\sum_{j=1}^{n}\frac{\beta_j\E[W_{j-1}]+a_m^2}{\psi_j}&=
E_1^2n^2 + 2E_1E_2n^{1+\Lambda}\\
&\qquad + \Bigl(\frac{B_1E_1}{2}+M E_1^2
+\frac{a_m^2+B_1E_3+E_1B_2-B_1E_1-B_1E_1M}{1-2\Lambda}\Bigr)n\\
& \qquad {} +o(n).
\end{split}
\end{equation}
Consequently, the first two terms in $\E[W_n^2]-(\E[W_n])^2$ cancel out. Only a leading linear term remains in the variance, and its coefficient is 
\[\Big(\frac{B_1E_1}{2}+M E_1^2
+\frac{a_m^2+B_1E_3+E_1B_2-B_1E_1-B_1E_1M}{1-2\Lambda}\Big)-2E_1E_3.
\]
The stated result follows after simplification aided by a computer algebra system and using the fact that $b_0=\sigma(1-\Lambda)-a_m$. 

For $\Lambda<0$ we can proceed in a similar way. The expansion~\eqref{ExpansionJ} is still valid. The only difference is that the magnitude of the error is larger 
and of order $O(n^{-\Lambda})$ in\eqref{ExpansionSmall1}. Nevertheless, the resulting expansion~\eqref{ExpansionSmall2} is still valid due to the multiplication with 
$\psi_n\sim n^{2\Lambda}$.

For $\Lambda=\frac12$, we proceed in a similar way. We use the identity
\begin{equation}
\sum_{j=1}^{n}j^{-1}=\ln n+\gamma+\OO\Bigl(\frac1 n\Bigr),
\label{KNUTH2}
\end{equation}
where $\gamma$ denotes the Euler-Mascheroni constant. We have $B_1E_1=2a_m \frac{a_m}{1-\Lambda}=\frac{a_m^2}{(1-\Lambda)^2}=E_1^2$, 
and also $2B_1E_2=2E_1E_2$, such that
\begin{align*}
\psi_n\sum_{j=1}^{n}\frac{\beta_j\E[W_{j-1}]+a_m^2}{\psi_j}&=
E_1^2n^2 + 2E_1E_2n^{1+\Lambda} \nonumber \\
&\quad + \big(a_m^2+B_1E_3+E_1B_2-B_1E_1-B_1E_1M\big)n\ln n+\OO(n).
\end{align*}
Consequently, the first two terms in $\E[W_n^2]-\E[W_n]^2$ cancel out again,
and the important constant is given by
\[
a_m^2+B_1E_3+E_1B_2-B_1E_1-B_1E_1M.
\]
The stated result is obtained after simplification.

For large-index urns $\Lambda>\frac12$ we cannot neglect errors of magnitude $\OO(1)$ as in the case $0<\Lambda<\frac12$. 
In order to deal with the cancellations, we adapt~\eqref{VarSimple1Eqn} and use a different exact representation
\begin{align*}
\E[W_n^2] &= \frac{\psi_n}{\psi_0}  W_0^2 + \psi_n\sum_{j=1}^{n}\frac{\beta_j\E[W_{j-1}]+a_m^2-\psi_jB_1j^{-\Lambda}(E_1j^{1-\Lambda}+E_2)}{\psi_j}\nonumber\\
   & \qquad {} +\psi_n B_1\sum_{j=1}^{n}(E_1j^{1-2\Lambda}+E_2j^{-\Lambda}).
\end{align*}
Owing to~\eqref{ExpansionJ} we know that the first sum is convergent by the comparison test. 
Application of~\eqref{KNUTH1} to the second sum gives
\[
\psi_n B_1\sum_{j=1}^{n}(E_1j^{1-2\Lambda}+E_2j^{-\Lambda})=
E_1^2n^2+ 2E_1E_2n^{1+\Lambda}+B_1\big(E_1\zeta(2\Lambda-1)+E_2\zeta(\Lambda)\big)n^{2\Lambda}+o(n^{2\Lambda}).
\]
The first two terms in $\E[W_n^2]- (\E[W_n])^2$ cancel out, 
and the constant $C$ is given by
\[
\frac{W_0^2}{\psi_0} +\sum_{j=1}^{\infty}\frac{\beta_j\E[W_{j-1}]+a_m^2-\psi_jB_1j^{-\Lambda}(E_1j^{1-\Lambda}+E_2)}{\psi_j}
+B_1\big(E_1\zeta(2\Lambda-1)+E_2\zeta(\Lambda)\big)-E_2^2,
\]
which proves the stated result.
\end{proof}

\smallskip

\subsection{Almost-sure convergence of nontriangular urns}
For triangular urns with $a_m=0$ (or $b_0=0$ or both) we have already obtained a limit theorem for $W_n$ via the martingale in Proposition~\ref{Prop:caligraphicW}. A first byproduct of the previous result concerning the first and second moment is a limit theorem for $W_n/T_n$ for $a_m\neq 0$ (and $b_0\neq 0$). 

\begin{prop}
Let $W_n$ be the number of white balls in the urn after $n$ draws. 
For nontriangular balanced affine urn models with $\Lambda<1$ the ratio of white balls $W_n$ over the total number $T_n=T_0+n\sigma$ after $n$ draws converges almost surely:
$$\frac{W_n}{T_n} \as   \frac{a_m}{\sigma(1-\Lambda)}=\frac{a_m}{a_m+b_0}.$$
\end{prop}
\begin{proof}
Following~\cite{ChenKu2013+} we use supermartingale theory to obtain the stated result. We only present the computation for \mor, the proof for \momp is very similar. The following computations are somewhat lengthy, and 
preferably carried out with the help of a computer algebra system.
Let $Z_n=\frac{W_n}{T_n}-  \frac{a_m}{\sigma(1-\Lambda)}$. 
Using Proposition~\ref{MuliDrawsLinPropLinear} we obtain 
\[
\E\bigl[ Z_n \given \field_{n-1}\bigr]= \Big(1+\frac{\sigma\Lambda}{T_{n-1}}\Big)\Big(1-\frac{\sigma}{T_{n}}\Big)Z_{n-1}
=\frac{T_{n-1}+\sigma\Lambda}{T_n}Z_{n-1}.
\]
Furthermore, in a manner similar to the proof of Proposition~\ref{MuliDrawsLinProp2}, we get 
\begin{equation*}
\begin{split}
\E\bigl[ Z_n^2 \given \field_{n-1}\bigr]&=\frac{T_{n-1}^2}{T_n^2}\Big(1+\frac{2\sigma\Lambda}{T_{n-1}}+\frac{\Lambda^2(m-1)\sigma^2}{mT_{n-1}^2}\Big)
Z_{n-1}^2\\
&\quad+\frac{\Lambda^2\sigma(\sigma(1-\Lambda)-2a_m)}{(1-\Lambda)mT_n^2}Z_{n-1}+\frac{a_m\Lambda^2(\sigma(1-\Lambda)-a_m)}{(1-\Lambda)^2mT_n^2}.
\end{split}
\end{equation*}
Hence,
\begin{equation*}
\begin{split}
\E\bigl[ Z_n^2 \given \field_{n-1}\bigr]&\le 
\frac{(T_{n-1}+\sigma\Lambda)^2}{T_n^2}Z_{n-1}^2+\frac{\Lambda^2\sigma(\sigma(1-\Lambda)-2a_m)}{(1-\Lambda)mT_n^2}Z_{n-1}+\frac{a_m\Lambda^2(\sigma(1-\Lambda)-a_m)}{(1-\Lambda)^2mT_n^2}.
\end{split}
\end{equation*}
Now we use the fact that $\sigma(1-\Lambda)-a_m=b_0\ge 0$ and also $a_m\ge 0$. 
Moreover, we know that $0\le \frac{W_n}{T_n}\le 1$, and consequently $-\frac{a_m}{\sigma(1-\Lambda)}\le Z_n\le 1-\frac{a_m}{\sigma(1-\Lambda)}$. Thus, there exists a constant $\kappa_1=\kappa_1(m,a_{m-1},a_m,\sigma)$---independent of $n$---such that
\begin{equation*}
\begin{split}
\E\bigl[ Z_n^2 \given \field_{n-1}\bigr]&\le \frac{(T_{n-1}+\sigma\Lambda)^2}{T_n^2}Z_{n-1}^2+\frac{\Lambda^2\sigma|(b_0-a_m)Z_{n-1}|+a_mb_0\Lambda^2}{(1-\Lambda)mT_n^2}\\
&\le \frac{(T_{n-1}+\sigma\Lambda)^2}{T_n^2}Z_{n-1}^2+\frac{\kappa_1}{T_n^2}.
\end{split}
\end{equation*}
There exists a constant $\kappa_2>0$ such that $\frac{\kappa_2}{T_n}+\frac{\kappa_1}{T_n^2}\le \frac{\kappa_2}{T_{n-1}}$.
Let $c_n=\frac{(T_{n-1}+\sigma\Lambda)^2}{T_n^2}$ with $0<c_n <1$ and $d_n=\frac{\kappa_1}{T_n^2}>0$.
We have 
\begin{equation*}
\E\bigl[ Z_n^2 +\frac{\kappa_2}{T_n}\given \field_{n-1}\bigr]\le c_n Z_{n-1}^2+d_n +\frac{\kappa_2}{T_n} \le Z_{n-1}^2+\frac{\kappa_1}{T_n^2}+ \frac{\kappa_2}{T_n} \le Z_{n-1}^2+\frac{\kappa_2}{T_{n-1}}.
\end{equation*}
Hence, $Z_n^2 +\frac{\kappa_2}{T_n}$ is a positive supermartingale, which converges almost surely. Thus, $Z_n^2$ converges almost surely.
Let $\lim_{n\to\infty} Z^2_n= Z$ almost surely. Following~\cite{ChenKu2013+}, we next prove that 
$\E[Z^2_n]\to 0$. By dominated convergence this is sufficient to obtain the stated result since it implies that $\E[Z] = 0$ and so $Z = 0$ almost surely, such that $Z_n$ converges to $0$ almost surely. 
We have
 \begin{equation*}
\E[ Z_n^2]\le c_n\E[Z_{n-1}^2]+d_n.
\end{equation*}
By the comparison theorem $$\sum_{n=1}^{\infty}d_n=\sum_{n=1}^{\infty}\frac{\kappa_1}{(n\sigma+T_0)^2}<\infty.$$  Moreover, 
$$\prod_{k=1}^{n}\frac{(T_{k-1}+\sigma\Lambda)^2}{T_k^2}=
\prod_{k=1}^{n}\frac{(k+\Lambda-1+\frac{T_0}\sigma)^2}{(k+\frac{T_0}\sigma)^2}=
\frac{\Gamma(n+\Lambda+\frac{T_0}\sigma)^2\Gamma(1+\frac{T_0}\sigma)^2}{\Gamma(\Lambda+\frac{T_0}\sigma)^2\Gamma(n+1+\frac{T_0}\sigma)^2}.$$
Thus, we can use the following lemma ---also used in~\cite{ChenKu2013+}--- to show that $\E[Z^2_n]\to 0$ and to finish our proof.
\begin{lemma}[\cite{ChenHsiauYang}]
Suppose $\{x_n\}_{n\geq 1}$, $\{c_n\}_{n\geq 1}$ and $\{d_n\}_{n\geq 1}$ are nonnegative real sequences satisfying $x_{n+1}\leq c_n x_n+d_n$, where $0<c_n<1$ for $n\geq 1$. If $\prod_{i=1}^n c_i\to 0$ and $\sum_{n=1}^\infty d_n<\infty$, then $x_n\to0$.
\end{lemma}
By Stirling's formula the product $\prod_{k=1}^{n}\frac{(T_{k-1}+\sigma\Lambda)^2}{T_k^2}$ 
satisfies the asymptotic expansions $$\prod_{k=1}^{n}\frac{(T_{k-1}+\sigma\Lambda)^2}{T_k^2}\sim \kappa_3 n^{-1+\Lambda},$$
for some constant $\kappa_3$. Since $\Lambda<1$ the product tends to zero, and so does
$\E[Z^2_n]$.
\end{proof}

\subsection{Almost-sure convergence for urns with large index}
\begin{theorem}
For nontriangular balanced affine urn models with a large index $\frac 1 2 < \Lambda < 1$ the random variable $\calW_n=g_n(W_n-\E[W_n])$ 
converges almost surely and in ${\rm L}_2$ to a limit $\calW_\infty$.
\end{theorem}
\begin{proof}
By Proposition~\ref{Prop:caligraphicW}, $\calW_n$ is a martingale. 
Hence, by martingale theory it suffices to prove that
\[
\sum_{n=1}^{\infty}\E\big[(\backward\calW_n)^2\big]<\infty
\]
in order to prove almost-sure and $\mathcal{L}_2$ convergence (see 
Chapter 10 of ~\cite{Williams}). 
We use a standard argument:
\[
\E\big[(\backward\calW_n)^2\mid  \field_{n-1}\big] =
\E\big[(\calW_n-\calW_{n-1})^2\mid  \field_{n-1}\big] =
\E\big[\calW_n^2-2\calW_n\calW_{n-1}+\calW_{n-1}^2\mid  \field_{n-1}\big].
\]
By the martingale property we get further
\[
\E\big[(\backward\calW_n)^2\mid  \field_{n-1}\big] =\E\big[\calW_n^2\mid  \field_{n-1}\big]-2\calW_{n-1}\E\big[\calW_n\mid \field_{n-1}\big]+\calW_{n-1}^2
=\E\big[\calW_n^2\mid  \field_{n-1}\big]-\calW_{n-1}^2.
\]
This implies that
\[
\E\big[(\backward\calW_n)^2\big]=\E\big[\calW_n^2\big]-\E\big[\calW_{n-1}^2\big],\quad n\ge 1.
\]
Using the fact that $\calW_0=0$ we obtain
\[
\sum_{n=1}^{N}\E\big[(\backward\calW_n)^2\big]=\E\big[\calW_N^2\big]=g_N^2 \V[W_N].
\]
By the asymptotic expansion~\eqref{ExpansionGn} and of $\V[W_n]$ 
we observe that 
\[
\lim_{N\to\infty} \sum_{n=1}^{N}\E\big[(\backward\calW_n)^2\big] = C\frac{\Gamma^2(\frac{T_0}\sigma+\Lambda)}{\Gamma^2(\frac{T_0}\sigma)}<\infty,
\]
with $C$ as given in Theorem~\ref{TheVar}.
\end{proof}

Some corollaries of the relatively small variance for $\Lambda\le \frac12$ are helpful in deriving
further distributional results.
\begin{coroll}
\label{Lm:Lconcentration}
Let 
$W_n$ be the number of white balls in the urn after $n$ draws, 
Then, for $\Lambda\le\frac12$, we have\footnote{By saying
a sequence of random variables $Y_n$
is $\OL  (g(n))$, we mean there exist a positive
constant $C$ and a positive integer $n_0$, such that $\E[|Y_n|] \le C |g(n)|$, for all $n \ge n_0$.}
$$W_n =  \frac{a_m}{1-\Lambda}\, n  +  \OL \bigr (\sqrt{	\V[W_n]}\, \bigl)
= \frac{a_m}{1-\Lambda}\, n  +
\begin{cases}
\OL (\sqrt{n}\, ),&\Lambda<\frac12;\\
\OL (\sqrt{n\ln n}\, ),& \Lambda=\frac12,\\
\end{cases}$$
and
$$W_n^2 = \Bigl(\frac{a_m}{1-\Lambda}\Bigr)^2\, n^2  
          +  \OL \bigl(	n^{\frac12}\V[W_n]\bigr) 
         =  \Bigl(\frac{a_m}{1-\Lambda}\Bigr)^2\, n^2   +
\begin{cases}
\OL (n^{\frac32}),\quad \Lambda<\frac12;\\
\OL (n^{\frac32}\ln n),\quad \Lambda=\frac12.\\
\end{cases}.$$
\end{coroll}
\begin{proof}
From the asymptotics of the mean and variance, as given in Proposition~\eqref{MuliDrawsLinProp1}, \eqref{ExpansionEn} and Theorem~\ref{TheVar}, for large~$n$
we have
\begin{align}
\E\Bigl[\Bigl(W_n -   \frac{a_m}{1-\Lambda} n \Bigr)^2\Bigr] 
          &= \E\Bigl[\Bigl(\bigl(W_n -  \E[W_n]\bigr) + \Bigl(\E[W_n] - \frac{a_m}{1-\Lambda}\, n\Bigr) \Bigr)^2\Bigr]\nonumber \\
          &=\V[W_n] + \Bigl(\E[W_n] -  \frac{a_m}
                    {1-\Lambda}\, n  \Bigr)^2\nonumber\\
          &= \Gro(\V[W_n]).
\label{Eq:Lonenorm}
\end{align}
So, by Jensen's inequality
$$\E\Bigl[\Bigl | W_n - \frac{a_m}{1-\Lambda}\, n  \Bigr |\Bigr]
\le \sqrt {\E\Bigl[\Bigl(W_n-  \frac{a_m}{1-\Lambda}\, n  \Bigr)^2\Bigr] }
 = \Gro (\V[W_n]),$$
and this implies  
$$W_n =   \frac{a_m}{1-\Lambda}\, n  + \OL(\V[W_n]).$$
The second part follows by squaring.
\end{proof}
\begin{coroll}
\label{Lm:convL}
Let $W_n$ be the number of white balls in the urn after $n$ draws.
Then, we have
$$\frac {W_n} n \inL \frac {a_m} {1 - \Lambda},$$ 
and
$$\frac {W_n^2} {n^2} \ \inL \Bigl(\frac{a_m}{1-\Lambda}\Bigr)^2.$$
So, both convergences occur in probability, too.  
\end{coroll}
\subsection{Martingale central limit theorem}
We follow the approach used in~\cite{KuMaPan2013+,Mah2012} used for special urns. 
We would obtain a Gaussian law for $ {\mathcal{W}_j}$, if a set of conditions for the martingale
central limit theorems are satisfied.
There is more than one such set (see~\cite{Hall}). 
One set of such conditions convenient in our work is the combined
conditional Lindeberg's condition and the conditional variance condition.
The conditional Lindeberg condition 
requires that, for some positive
increasing sequence~$\xi_n$, and for all $\varepsilon>0$,
$$U_n := \sum_{j=1} ^n \E\Bigl[\Bigl(\frac{\nabla {\mathcal{W}}_j}
         {\xi_n}\Bigr)^2 \indicator \Bigl(\Big|
                   \displaystyle \frac{\nabla {\mathcal{W}}_j}
             {\xi_n}\Bigr| > \varepsilon\Bigr)\Bigiven \field_{j-1}\Bigr]\inprob 0,$$
and the conditional variance condition requires that, for some square integrable 
random variable $Y \neq 0$, we have
\begin{align*}
V_n: =\sum^n_{j=1}\E\Bigl[\Bigl(\frac{\nabla {\mathcal{W}}_j}{\xi_n}\Bigr)^2 \Bigiven \field_{j-1}\Bigr]\inprob Y.
\end{align*}
When these conditions are satisfied, we get
$$\frac {W_n - \frac{a_m} {1 - \Lambda} \, n} {\xi_n} \ \convD \ \normal (0, Y),$$
where the right-hand side is a mixture of normals, with 
$Y$ being the mixer.
It will turn out that the correct scale factors are 
\begin{equation}
\xi_n=
\begin{cases}
n^{\frac 1 2 - \Lambda},\quad\text{for }\Lambda<\frac12,\\
\ln n,\quad\text{for }\Lambda=\frac12.
\end{cases}
\label{eq:}
\end{equation}
\begin{lemma}
\label{Lm:uniformbound}
The terms $|\nabla {\mathcal{W}}_j| $ satisfy $|\nabla \mathcal{W}_j|\le 
K j^{-\Lambda}$
for some positive constant $K$ and $1\le j \le n$.
\end{lemma}
\begin{proof}
Suppose $\omega_j=W_j-W_{j-1}$ is the random number of white 
balls added right after the $j$th drawing. Starting from the definition of $\mathcal{W}_j$, we write the absolute difference as
\begin{align*}
|\nabla {\mathcal{W}}_j| &= \ | {\mathcal{W}_j}  -  {\mathcal{W}_{j-1}}  | \\ 
			&= \Big| \Bigl(g_j W_j - a_m \sum_{k=1}^{j} g_k -W_0\Bigr) - \Bigl(g_{j-1} W_{j-1} - a_m \sum_{k=1}^{j-1} g_k -W_0\Bigr)\Big|\\
    &= \big| g_j (W_{j-1} + \omega_j) - a_m g_{j} -g_{j-1} W_{j-1}\big|\\
     &= \big| W_{j-1}  \nabla g_j  + g_j \omega_j - a_m g_{j} \big|\\
    &\le  T_{j-1} \, | g_j - g_{j-1}  |+ q  g_j  + q g_{j-1},
\end{align*}
with $q = \max_{0\le k\le m} |a_k|$. By definition of $g_n$ and the asymptotic expansion~\eqref{ExpansionGn} $g_j=\OO(j^{-\Lambda})$ and further $| g_j - g_{j-1} |=\OO(j^{-1-\Lambda})$. 
Consequently, there exists a constant $K>0$ such that 
$$
|\nabla {\mathcal{W}}_j|  \le K j^{-\Lambda}
$$
\end{proof}
\begin{lemma}
\label{Lm:Un}
$$U_n = \sum^n_{j=1}\E\Bigl[\Bigl(\frac{\nabla\mathcal{W}_j}{\xi_n}\Bigr)^2{\indicator} \Bigl(\Big|
                   \displaystyle \frac{\nabla\mathcal{W}_j}
             {\xi_n}\Bigr| > \varepsilon\Bigr)\Bigiven 
             \field_{j-1}\Bigr]\inprob 0.$$
\end{lemma}
\begin{proof}
Choose any $\varepsilon > 0$. Concerning $\Lambda<\frac 1 2$
we distinguish between $\Lambda<0$ and $0<\Lambda<\frac 1 2$. Lemma~\ref{Lm:uniformbound} asserts that for arbitrary $j$ with $1\le j\le n$
$$
 \frac{\nabla {\mathcal{W}}_j}
             {\xi_n} \le 
              \begin{cases}
              \frac{K}{n^{\frac12}},& \text{for }\Lambda<0;\\
              \frac{K}{j^{\Lambda} n^{\frac 1 2 - \Lambda}} \le \frac{K}{n^{\frac12-\Lambda}} , & \text{for }0<\Lambda<\frac12;\\
              \frac{K}{j^{\frac12}\ln n}\le  \frac{K}{\ln n},  &\text{for }
                \Lambda=\frac12.\\
             \end{cases}
$$
Hence, the sets $\{|\backward \mathcal{W}_j |> \varepsilon \, \xi_n\}$ are all empty, regardless of $\Lambda<0$, $0<\Lambda<\frac12$ or $\Lambda =\frac12$,
for all $n$ greater than some positive integer $n_0(\varepsilon)$. For large $n$ (namely, $n > n_0(\varepsilon)$), we can
stop the sum at $n_0(\varepsilon)$. By in Lemma~\ref{Lm:uniformbound},
we get
\begin{align*}
U_n &= \sum_{j=1}^{n_0(\varepsilon)} \E\Bigl[\Bigl(\frac { \backward\mathcal{W}_j}
     {\xi_n }\Bigr)^2
      \indicator \Bigl(\Big| \displaystyle \frac {\backward\mathcal{W}_j} {\xi_n} \Big| >
      \varepsilon\Bigr)
                     \,  \Big |\, \field_{j-1}\Bigr] \le  \frac 1 {\xi_n^2} \sum_{j=1}^{n_0(\varepsilon)} 
                 \E\Bigl[K^2
                \, \Big | \,  \field_{j-1}\Bigr] \le  \frac {K^2 n_0(\varepsilon)}  {\xi_n^2} \to 0,
\end{align*}
for $n\to\infty$.
\end{proof}
\begin{lemma}
\begin{align*}
V_n =\sum^n_{j=1}\E\Bigl[\Bigl(\frac{ \nabla\mathcal{W}_j}{\xi_n}\Bigr)^2 \Bigiven \field_{j-1}\Bigr]\inprob \frac{a_mQ^2(\sigma(1-\Lambda) - a_m)\Lambda^2} {m(1-\Lambda)^2(1 -2 \Lambda)}.
\end{align*}
\end{lemma}
\begin{proof}
Let $Q := \frac {\Gamma(\frac{T_0} \sigma + \Lambda)} {\Gamma(\frac {T_0} \sigma)}.$
By~\eqref{ExpansionGn} we have 
$$g_n = Q n^{-\Lambda} + O (n^{-\Lambda-1}).$$
From this asymptotic relation, we also have
$$\nabla g_n = - Q \Lambda n^{-\Lambda-1} + O (n^{-\Lambda-2}).$$
As in the proof of Lemma~\ref{Lm:Un}, we write
$$ {\nabla\mathcal{W}}_j 
       = (\nabla g_j) W_{j-1} + g_j \omega_j                   
                -a_m g_{j-1}.$$
And so, we can write    
$$\bigl( {\nabla\mathcal{W}}_j\bigr)^2
     = \frac {Q^2}   {j^{2\Lambda}}\Bigl(\Lambda^2\frac {W_{j-1}^2} {j^2} -2\Lambda\omega_j\frac{W_{j-1}} j+2\Lambda a_m \frac {W(j-1)} j 
                  + \omega_j^2 -2a_m\omega_j + a_m^2 \Bigr)  .$$      
Using the L${}_1$ approximation in Corollary~\ref{Lm:convL}, we write
the conditional expectation
\begin{align*}
\E\bigl[\bigl( {\nabla\mathcal{W}}_j\bigr)^2 \given \field_{j-1}\bigr]
    & = \frac {Q^2}   {j^{2\Lambda}}\Bigl(\Lambda^2 \Bigl(\frac {a_m} {1-\Lambda}\Bigr)^2 
                 - 2\Bigl(\Lambda\Bigl(\frac {a_m} {1-\Lambda}\Bigr)+ a_m\Bigr)\E\bigl[\omega_j\given \field_{j-1}\bigr]  
                         + 2\Lambda a_m \Bigl(\frac {a_m} {1-\Lambda}\Bigr) \\
           &\qquad {} + \E\bigl[\omega_j^2\given \field_{j-1}\bigr]  + a_m^2 \Bigr) + \Gro_{L_1}\Bigl( \frac 1 {j^{2\Lambda+1}}\Bigr) .
\end{align*}
We already know exactly the conditional expectations of $\omega_j$ and $\omega_j^2$ from~\eqref{MuliDrawsOmegaJEins} and
~\eqref{MuliDrawOmegaJZwei1}, respectively:
\[
\E\bigl[\omega_j\given \field_{j-1}\bigr] = \frac{\sigma\Lambda}{T_{j-1}}W_{j-1}+a_m
\]
and 
\begin{equation*}
\begin{split}
\E\bigl[\omega_j^2\given \field_{j-1}\bigr] &=
\E[W_j^2\given \field_{j-1}]-2W_{j-1}\E[W_j\given \field_{j-1}]+W_{j-1}^2\\
&=\big(\alpha_j-2(1+\frac{\sigma\Lambda}{T_{j-1}})+1\big)W_{j-1}^2+(\beta_j-2a_m)W_{j-1}+\gamma_j,
\end{split}
\end{equation*}
with model-dependent sequendes $\alpha_j$, $\gamma_j$ as given in~\eqref{MuliDrawOmegaJZwei1} and $\beta_j$ in Proposition~\ref{MuliDrawsLinProp2}.
Consequently, using asymptotic expansions of $\alpha_j$, $\beta_j$, $\gamma_j$, and Corollary~\ref{Lm:convL} we obtain the following model-independent expansions  
\[
\E\bigl[\omega_j\given \field_{j-1}\bigr] =\frac{a_m}{1-\Lambda}+\Gro_{L_1}\Bigl( \frac {1}{ \sqrt{j}}\Bigr),
\] 
and
\[
\E\bigl[\omega_j^2\given \field_{j-1}\bigr] =\frac{a_m(\Lambda^2(\sigma-\Lambda\sigma-a_m)+ma_m)}{m(1-\Lambda)^2}+\Gro_{L_1}\Bigl(  \frac {1}{ \sqrt{j}}\Bigr).
\] 
Putting all the elements together and simplifying, we obtain
$$\E\bigl[\bigl( {\nabla\mathcal{W}}_j\bigr)^2 \given \field_{j-1}\bigr] = \frac{a_m Q^2(\sigma(1-\Lambda) - a_m)\Lambda^2} {m j^{2\Lambda}(1-\Lambda)^2} +  \Gro_{L_1}\Bigl( \frac 1 {j^{2\Lambda+1}}\Bigr).$$ 
Now we can sum using \eqref{KNUTH1} and \eqref{KNUTH2}. We get for $\Lambda<\frac12$ 
$$
V_n =\frac{a_m Q^2(\sigma(1-\Lambda) - a_m)\Lambda^2} {m(1-\Lambda)^2(1 -2 \Lambda)} +  \Gro_{L_1}\Bigl( \frac 1n\Bigr) \inL \frac{a_m Q^2(\sigma(1-\Lambda) - a_m)\Lambda^2} {m(1-\Lambda)^2(1 -2 \Lambda)},
$$
and, for $\Lambda=\frac12$ we get,
$$
V_n =\frac{a_m Q^2(\frac12\sigma - a_m)} {m} +  \Gro_{L_1}\Bigl( \frac 1{\log n}\Bigr) \inL \frac{a_m Q^2(\frac12\sigma - a_m)} {m}. 
$$
This implies the required convergence in probability.
\end{proof} 
Having checked the two martingale conditions, a Gaussian law follows for the nondegenerate cases:
$$\frac {\mathcal {W}_n - \frac {Q a_m} {1-\Lambda}\, n} {n^{\frac 1 2 - \Lambda}}
       \ \convD\ \normal\Bigl(0, \frac{a_m (\sigma(1-\Lambda) - a_m) Q^2\Lambda^2} {m(1-\Lambda)^2(1 -2 \Lambda)} \Bigr)$$
for $\Lambda<\frac12$ and 
$$
\frac {\mathcal {W}_n - 2Q a_m\, n} {\sqrt{\log n}}
       \ \convD\ \normal\Bigl(0, \frac{a_m (\frac12\sigma - a_m) Q^2} {m} \Bigr).
       $$
for $\Lambda=\frac12$. Translating this into a statement on the number of white balls and using $b_0 =\sigma(1-\Lambda) - a_m$ we get a main result of this investigation.
\begin{theorem}
Suppose we have a two-color tenable affine balanced urn that grows by sampling sets of size $m$
with or without replacement, and with a small index $\Lambda \le \frac 1 2$, and does not fall in any of the afore-mentioned degenerate cases.\footnote{Recall that several degenerate cases are excluded from this study, namely, the triangular case where $a_0 = 0$ or $b_0 = 0$, the zero-balance cases, and the case $T_0 +m (a_{m-1}-a_m) \le 0$.}  
Let $W_n$ be the number of white balls after $n$ draws.
For $\Lambda\le\frac12$ we have Gaussian laws:
$$\frac {W_n - \frac {a_m} {1-\Lambda}\, n} {\sqrt n}
       \ \convD\ \normal\Bigl(0,  \frac{a_m b_0 \Lambda^2} {m(1-\Lambda)^2(1-2\Lambda)}\Bigr ),\quad \text{for }\Lambda<\frac12,
       $$      
and
$$\frac {W_n - 2a_m\, n} {\sqrt{n\log n}}
       \ \convD\ \normal\Bigl(0,  \frac{a_m b_0 } {m}\Bigr ), \quad \text{for }\Lambda=\frac12.$$          
\end{theorem}
\section{Conclusion and Outlook}
\subsection{Summary}
We studied for a two-color affine linear urn models with multiple drawings---sample size $m\ge 1$---under two sampling models the distribution of the number of white balls $W_n$ after $n$ draws.
In the following we summarize our findings according to the index $\Lambda$ and state the order of growth of the expectation and variance.
\begin{table}[!htb]
\begin{tabular}{|c|c|c|c|}
\hline
\tableT\tableB &$\Lambda<\frac12$& $\Lambda=\frac12$ & $\Lambda>\frac12$\\[1ex]
\hline
\tableT\tableB $\E[W_n]$& $ n$ & $ n$  & $n$\\
\hline
\tableT\tableB $\V[W_n]$& $ n$ & $ n\log n$  & $n^{2\Lambda}$\\
\hline
\tableT\tableB Limit law &  $\frac{W_n-\E[W_n]}{\sqrt{\V[W_n]}}\to\normal(0,1)$ & $\frac{W_n-\E[W_n]}{\sqrt{\V[W_n]}}\to\normal(0,1)$ & $\mathcal{W}_n\as\mathcal{W}_\infty$\\
\hline
\end{tabular}

\bigskip
\caption{Overview of the result for nontriangular
urns $a_m b_0\neq 0$.}
\label{Summary1}
\end{table}

Here $\mathcal{W}_n=g_n\big(W_n-\E[W_n]\big)$ with $g_n=\prod_{j=0}^{n-1}\frac{T_j}{T_j+m(a_{m-1}-a_m)}\sim C n^{-\Lambda}$. Note that for non-normal limit law for large-index urns the distribution will depend on the sampling model; this will be discussed in a companion work, as well as the moment structure.
\subsection{Quadratic expected value and beyond}
Using Lemma~\ref{MuliDrawsLinLemma1} it is possible to extend Proposition~\ref{MuliDrawsLinPropLinear}
to characterize all ball replacement matrices leading to a conditional expected value 
of quadratic type, cubic type, etc., and in general to a polynomial of degree $k$, with $1\le k \le m$.
Beginning with the extension to quadratic types, 
$$\E\bigl[W_n \given \field_{n-1}\bigr]= \alpha_{n,2} W_{n-1}^2+ \alpha_{n,1} W_{n-1} +\alpha_{n,0},\qquad n\ge 1,$$
we obtain the same condition for the $a_k$'s for both models but different resulting coefficients $\alpha_{n,0}$, $\alpha_{n,1}$ and $\alpha_{n,2}$.
It is possible to obtain a explicit expression for the expected value of $W_n$, but the arising formula is very complicated 
and does not easily seem to lead to precise asymptotic expansions. 
\subsection{Urns with a large index, triangular urns, and more colors}
In the companion work~\cite{KuMaPan201314}, we complete the study of balanced urn models with multiple drawings and affine conditional expectation.  In particular, we provide a detailed analysis of the moments of $W_n$ for  triangular urn models and also for urns with a large index $\Lambda> 1/2$. The analysis is based on the so-called method of moments applied to the centered moments of $W_n$ and the martingales $\mathfrak{W}_n$, $\mathcal{W}_n$. 

\smallskip

In order to generalize the affinity condition of Proposition~\ref{MuliDrawsLinPropLinear} to more than two colors it is beneficial to rewrite the $a_k$'s as an affine combination of $a_0$ and $a_m$: $a_k=\frac{m-k}{m}a_0+ \frac{k}{m}a_m$, $0\le k\le m$.
This idea can be readily generalized to $r\ge 2$ colors. One obtains a martingale of the form
\[
\E[\mathbf{X}_n \mid \mathcal{F}_{n-1}]=({\bf I}+ \frac{1}{T_{n-1}} \tilde{\bf A}^T)\, \mathbf{X}_{n-1},
\]
with $\mathbf{X}_n=(X^{(1)}_n,\dots, X^{(r)}_n)^T$, where $X^{(\ell)}_n$ denotes the random number of balls colored $\ell$ and $\bf I$ the identity matrix. The matrix $\tilde{\bf A}^T$ is a certain $r\times r$-matrix (somewhat similar to the ``reduced'' ball replacement matrix $\matA$ introduced before), being composed of $r$ vectors appearing in the general affinity condition. Compared to the two color case, simple expressions for the (mixed) moments of $\mathbf{X}_n$ do not seem to exist, but it may be possible to study the limitings distribution of~$\mathbf{X}_n$ using different methods. 

\end{document}